\documentclass[12pt,reqno]{article}
\usepackage{color}
\usepackage{float}
\usepackage{lineno} 
\usepackage{comment}
\usepackage{graphics,amsmath,amssymb}
\usepackage{amsthm}
\usepackage{amsfonts}
\usepackage{latexsym}
\usepackage{epsf}
\usepackage{hyperref} 
\usepackage{amssymb}
\usepackage{graphicx}

\newtheorem{theorem}{Theorem}[section]
\newtheorem{corollary}[theorem]{Corollary}
\newtheorem{claim}[theorem]{Claim}
\newtheorem{proposition}[theorem]{Proposition}
\newtheorem{question}[theorem]{Question}

\newtheorem{problem}[theorem]{Problem}

\theoremstyle{definition}

\newtheorem{example}[theorem]{Example}
\newtheorem{remark}[theorem]{Remark}
\newtheorem{algorithm}{Algorithm}

\def\e{\epsilon}
\def\s{\sigma}
\def\t{\tau}
\def\b{\beta}
\def\a{\alpha}
\def\lam{\lambda}
\def\id{\mathop{id}}
\def\inv{^{-1}}
\def\cen{C_{S_n}}
\def\ms{\mathcal{S}}
\def\mc{\mathcal{C}}

\def\set{{\rm set}}

\title{Blocks in cycles and $k$-commuting permutations}

\author{Rutilo Moreno~ and Luis Manuel Rivera}
\date{ }
\begin{document}
\maketitle

\begin{abstract}
Let $k$ be a nonnegative integer, and let $\a$ and $\b$ be two permutations of $n$ symbols.  We say that $\a$ and $\b$ $k$-commute if $H(\a\b, \b\a)=k$, where $H$ denotes the Hamming metric between permutations. In this paper, we consider the problem of finding the permutations that $k$-commute with a given permutation. Our main result is a characterization of permutations that $k$-commute with a given permutation $\b$ in terms of blocks in cycles in the decomposition of $\b$ as a product of disjoint cycles. Using this characterization, we provide formulas for the number of permutations that $k$-commute with a transposition, a fixed-point free involution and an $n$-cycle,  for any $k$. Also, we determine the number of permutations that $k$-commute with any given permutation,  for $k \leq 4$.
\end{abstract}


\tableofcontents


\section{Introduction}

The symmetric group as a metric space has been studied with different metrics and for different purposes, see, for example, \cite{dezahuang, diaconis, far, schia}, and the metric that seems to be more used is the Hamming metric. This metric was introduced in 1950 by R. W. Hamming \cite{hamming} for the case of  binary strings and in connection with digital communications. For the case of permutations it was used in an implicit way by H. K. Farahat~\cite{far} who studied the symmetries of the metric space $(S_n, H)$, where $S_n$ denotes the symmetric group on $[n]:=\{1, \dots  ,n\}$ and $H$ the Hamming metric between permutations. Also, D. Gorenstein, R. Sandler and W. H. Mills \cite{gor} studied a problem about permutations that almost commute, in the sense of normalized Hamming metric.  Other problems studied in this metric space are the packing and
covering problem (see, e.g., \cite{quis}), and also permutation codes (see, e.g., \cite{came}) which have turned out to be useful in applications to power line communications (see, e.g. \cite{chu}). 

Our interest in the symmetric group as a metric space, with the Hamming metric, arose from the study of  sofic groups, a class of groups of growing interest (see, e.g.,  \cite{cag, pest}) that was first defined by M. Gromov~\cite{gromov} as a common generalization of residually finite groups and of amenable groups. To the best of the author's knowledge it is an open question to determine if all groups are sofic. The following, Theorem 3.5 in \cite{pest}, shows  the importance of the  Hamming metric $H$ on the symmetric group. Let $H_n(\a, \b)=H(\a, \b)/n$, for every $\a, \b \in S_n$.
\begin{theorem}\label{sofic} A group $G$ is sofic if and only if for evey finite $F \subseteq G$ and for each $\varepsilon >0$, there exist a natural $n$ and a mapping $\theta: F \rightarrow S_n$ so that 
\begin{enumerate}
\item if $g, h, gh \in F$, then $H_n\left(\theta(g)\theta(h), \theta(gh)\right) < \varepsilon$,
\item if the identity $e$ of $G$ belongs to  $F$ then $H_n(\theta(e), \id) \leq \varepsilon$, and
\item for all distinct $g, h \in F$, $H_n(\theta(g), \theta(h)) \geq 1/4$.
\end{enumerate}
\end{theorem}

As $\theta$ is not necessarily a group homomorphism between $G$ and $S_n$ then permutations $\theta(gh)$ and $\theta(g)\theta(h)$ can be different,  and condition $(1)$ in previous theorem ask for a ``small"  difference between them. Motivated by this, L. Glebsky and the second author defined the concept of stability of a system of equations in permutations \cite[Def. 1]{glebriv1}. For the convenience of the reader, we
remember here some of these definitions.   Let $w(x_1,\dots,x_k)=x_{i_1}^{\varepsilon_1}x_{i_2}^{\varepsilon_2} \dots  x_{i_l}^{\varepsilon_l}$, $u(x_1,\dots,x_k)=x_{j_1}^{\varepsilon'_1}x_{j_2}^{\varepsilon'_2} \dots  x_{j_m}^{\varepsilon'_l}$ be expressions using $x_h \in \{x_1, \dots, x_k\}$ and  $\varepsilon_i= \pm 1$,  $\varepsilon'_j= \pm 1$ (we may think
that $w, u$ are words in $\{x_1,x_1^{-1},\dots,x_k, x_k^{-1}\}$). We say that permutations $\a_1,\dots, \a_k$ are an {\it $\e$-solution} of 
equation  $w(x_1,\dots,x_k)=u(x_1,\dots,x_k)$,
if and only if  $H_n(w(\a_1,\dots,\a_k), u(\a_1,\dots,\a_k))\leq\e$. 
 We say that permutations $\a_1,\dots,\a_k$ are an {\it $\e$-solution} of a system of equations
\begin{equation}\label{eq2}
w_i(x_1,\dots,x_k)=u_i(x_1,\dots,x_k),\;\; i=1,\dots,r
\end{equation}
 if and only if $\a_1,\dots,\a_k$ are an 
$\e$-solution for each equation of the system. The
  system of equations (\ref{eq2}) is called {\it stable in permutations} if and only if there exists
$\delta_\e $, $\lim\limits_{\e\to 0}\delta_\e=0$, such that for any
$\e$-solution $\a_1,\dots, \a_k\in S_n$ of system~(\ref{eq2}), 
there exists an exact solution $\tilde \a_1,\dots,\tilde
\a_k\in S_n$ of ~(\ref{eq2}) such that $H_n(\a_i,\tilde
\a_i)\leq\delta_\e$ for $i=1,\dots,k$. 

Glebsky et. al., \cite{glebriv1}  
 showed that the stability in permutations of system~(\ref{eq2}) is related with the properties of the finitely presented group 
\begin{equation}\label{group1}
 G=\langle x_1, \dots  , x_k \;|\; w_i(x_1, \dots  , x_k)=u_i(x_1, \dots  , x_k), i=1, \dots , r \rangle.
\end{equation}They proved that if $G$ is finite then system~(\ref{eq2}) is stable in permutations, and if $G$ is sofic but not residually finite then system~(\ref{eq2})  is unstable in permutations. One question that remains open is to determine if for any finitely presented residually finite group $G$, with presentation (\ref{group1}), system~(\ref{eq2})  is stable or not in permutations. As the group  $G=\langle x, y\;|\; xy=yx\rangle$ is an example of a residually finite group we are interested in the following:
 \begin{problem}\label{problem1}
To determine if equation $xy=yx$ is or not stable in permutations.
   \end{problem}
In an informal way, this problem can be expressed as the following  ``almost" implies ``near" type problem (see, e.g.,  \cite{ander} for more problems of this type):   is it true that any pair of almost commuting permutations is closed to a commuting pair of permutations?  A problem, which is closely related to problem~\ref{problem1} was studied by D. Gorenstein, et al., \cite{gor}.

The analogous problem about the stability of $xy=yx$ in matrices is a classical problem in linear algebra and operator theory, and has been widely studied for the cases when the distance between matrices are: the operator norm (see, e.g.,  \cite{friIs, has, hlin, osborne, voicu}),  the Schatten norm (see, e.g., \cite{filo2}), and for the normalized Hilbert-Schmidt distance (see, e.g., \cite{filo, gleb, had2, had3}). On the other hand, the case of rank distance is poorly studied and little understood (see, e.g., \cite[Sec. 2.3]{glebriv2}). 

In order to get insight and to develop tools towards a solution  of Problem~\ref{problem1}, we begin the study of the following problems: let $\a$ and $\b$ be two permutations, we say that $\a$ and $\b$ {\it $k$-commute} if $H(\a\b, \b\a)=k$.

\begin{problem}\label{problem2}
 For given $\b \in S_n$, to characterize the permutations $\a$ that $k$-commute with $\b$.
 \end{problem}
 \begin{problem}\label{problem3}
  To compute the number $c(k, \b)$ of permutations that $k$-commute with $\b$, where $\b$ is any permutation and $k$ any nonnegative integer.
  \end{problem}
Our main result about Problem~\ref{problem2}  is a characterization of permutations $\a$ that $k$-commute with a given permutation $\b$. This characterization is given in terms of blocks formed by strings of consecutive points in the cycles of the decomposition of $\b$ as a product of disjoint cycles. 
With respect to Problem~\ref{problem3}, using our characterization we were able to find explicit formulas for $c(k,\beta)$, for any $\b$ and $k\leq 4$. The study of this small cases sheds light of how difficult it can be the problem of computing $c(k, \b)$ in its generality. So we worked with several specific types of permutations. We have found some relations between $c(k, \b)$ and the following integer sequences in OEIS \cite{oeis}: A208529, A208528 and A098916 when $\b$ is a transposition, A000757 when $\b$ is an $n$-cycle, and A053871 when $\b$ is a fixed-point free involution.  The relationship between the number $c(k, \b)$ with some integers sequences in OEIS have provided another motivation for the authors to studied permutations that $k$-commute using the Hamming metric. The interested reader in a similar problem but with strings is referred to the work of  J. Shallit~\cite{js}. 

The outline of the paper is as it follows. In Section~\ref{sec:preliminar} we give some of the definitions and notation  used throughout the paper. In Section~\ref{perthatkcommute} we present our characterization of permutations that $k$-commute with a given permutation $\b$. This characterization is given in terms of blocks in cycles in the decomposition of $\b$ as a product of disjoint cycles. In Section~\ref{numberckb} we present a formula and a bivariate generating function for the number of permutations that $k$-commute with any $n$-cycle. Also we present a result about the proportion of even permutations that $k$-commute with $\b$. In Section~\ref{ckb34} we obtain explicit formulas for the number $c(k, \b)$ when $\b$ is any permutation and $k=3, 4$. In Section~\ref{involutions} we obtain some formulas for the cases when $\b$ is a transposition and a fixed-point free involution. 

\section{Definitions and notation}\label{sec:preliminar}
We first give some definitions and notation used throughout the work. The elements in $[n]$ are called {\it points} and the elements in $S_n$ are called permutations or $n$-permutations. For any permutation $\pi \in S_n$, we write $\pi =p_1 p_2 \dots p_n $ for its one-line notation, i.e., $\pi(i)=p_i$ for every $i \in [n]$. We compute the product $\a\b$ of permutations $\a$ and $\b$ by first applying $\b$ and then $\a$.  
 A permutation $\pi \in S_n$ is called a {\it cycle} of length $m$ (or $m$-{\it cycle}), and it is denoted by $\pi=(a_1 a_2 \dots a_m)$, if $\pi(a_i)=a_{i+1}$, for $1\leq i < m$, $\pi(a_m)=a_1$ and $\pi(a)=a$ for every $a \in [n] \setminus \{a_1, \dots , a_m\}$.   
It is a known fact that any permutation can be written in essentially one way as a product of disjoint cycles (called its {\it cycle decomposition}, see, e.g., \cite[Sec.\ 1.3, p.\ 29]{dum}). In this paper, we will denote a cycle in the disjoint cycle decomposition of $\pi$ by $\pi_j$, i.e., $\pi$ with a  subindex $j \in [n]$ that not necessarily means its length, and we will say that $\pi$ {\it has cycle} $\pi_j$ or that $\pi_j$ is {\it a cycle of} $\pi$.  If $\pi_j=(a_1 \dots  a_m)$ is a cycle of $\pi$ we define $\set(\pi_j):=\{a_1, \dots , a_m\}$, and we say that $a$ is a point in cycle $\pi_j$ if $a \in \set(\pi_j)$. The {\it cycle type} of a permutation $\b$ is a vector $(c_1, \dots , c_n)$ that indicates that $\b$ has exactly $c_i$ cycles of length $i$ in its cycle decomposition.  
The {\it Hamming metric}, $H(\a, \b)$, between permutations $\a$ and $\b$  is defined as $H(\a, \b)=|\{ a \in [n] \; : \; \a(a) \neq \b(a)\}|$.
 It is well-known (see, e.g.,  \cite{dezahuang}) that this metric is bi-invariant, that not two permutations have Hamming metric equal to 1, and also, that $H(\a, \b)=2$ if and only if $\a\b^{-1}$ is a transposition. We say that  $a\in[n]$ is a {\it good commuting point}  (resp. {\it bad commuting point}) of $\a$ and $\b$ if $\a\b(a)= \b\a(a)$ (resp. $\a\b(a)\neq \b\a(a)$). 
Usually, we abbreviate good commuting points (resp. bad commuting points) with {\it g.c.p.} (resp. {\it b.c.p.}). In this work, we use the convention $m \bmod m = m$  for any  positive integer $m$.

\subsection{Blocks in cycles}
Let $\pi \in S_n$, a {\it block} $A$ in a cycle $\pi_j=(a_1a_2\dots a_m)$ of $\pi$ is a nonempty string $A=a_ia_{i+1} \dots  a_{i+l}$, where $l \leq m$, of consecutive elements  in $\pi_j$, where the sums on the subindex are taken modulo $m$. This definition is different from the given in \cite{gor}, our way of defining a block in a cycle is similar to the definition of block when permutation is written in one-line-notation (see, e.g., (\cite{bona3, chris}). The {\it length} of a block $A$ is the number of elements in the block, and is denoted by $|A|$. If we have the block $A=a_1 \dots  a_l$, the elements $a_1$ and $a_l$ are called the {\it first} and the {\it last} elements of the block, respectively. A {\it proper block} (resp. {\it improper block}) of an $m$-cycle is a block of length $l < m$ (resp. $l=m$).  
 Two blocks $A$ and $B$ are said to be {\it disjoint} if they do not have points in common. 
The {\it product} $AB$ of two disjoint blocks, $A$ and $B$, not necessarily from the same cycle, is defined by concatenation of strings (this product is not necessarily a block in a cycle of $\b$). If $\pi_j$ is a cycle, we sometimes write $\pi_j=(A_1\dots A_k)$ to mean that one of the $m$ cyclic equivalent ways to write $\t$ is equal (as a block) to $A_1\dots A_k$. A {\it block partition} of a cycle $\pi_j$ is a set $\{A_1, \dots, A_l\}$ of pairwise disjoint blocks such that there exist a product $A_{i_1}\dots  A_{i_l}$ of these blocks such that $\pi_j=(A_{i_1}\dots  A_{i_l})$.
If $P= J_{1}J_{2}\dots  J_{k}$ is a block product of $k$ pairwise disjoint blocks (not necessarily from the same cycle) and $\a$ is a permutation in $S_k$, the {\it block permutation} $\phi_\a$ induced by $\a$ and $P$ is defined as $\phi_\a(P)=J_{\a(1)}J_{\a(2)}\dots  J_{\a(k)}$.  
\begin{example}
Let $\pi=\pi_1\pi_2 \in S_9$ with $\pi_1=(1234), \pi_2=(56789)$. An improper block in $\pi_1$ is $2341$ and a proper block is $A=12$. The blocks  $B_1=567$, $B_2=8$, $B_3=9$ are a block partition of $\pi_2$. The product $B_1B_2$ is a block in $\pi_2$ and $AB_2=128$ is not a block in any cycle of $\pi$. Let $\a=(321) \in S_3$. The block permutation $\phi_\a(B_1B_2B_3)$ is $B_3B_1B_2=95678$.  
\end{example} 

Let $\a, \b \in S_n$ and $\b_j=(b_1\dots  b_m)$ a cycle of $\b$. It is well known (see, e.g., \cite[Prop.\ 10, p.\ 125]{dum}) that $\a \b_j \a^{-1}=(\a(b_1) \dots \a(b_m))$, i.e., $\a\b_j\a^{-1}$ is also an $m$-cycle, not necessarily of $\b$. Sometimes we  write $\a|_{\set(\b_j)}$, the restriction of $\a$ to $\set(\b_j)$, as
\begin{equation}\label{alfares}
\a|_{\set(\b_j)}=\left(
\begin{array}{ccccccccc}
 b_1 & b_2 & \dots   & b_m \\
\a(b_1) & \a(b_2) & \dots   & \a(b_m)
 \end{array}
\right).
\end{equation}

If $\a|_{\set(\b_j)}$ is written as in (\ref{alfares}), we will write 
\[
\a|_{\set(\b_j),  k}=\left(\begin{array}{cccc}
B_1 B_2  \dots   B_k   \\
 J_1  J_2   \dots   J_k  
 \end{array} \right),
\]  
to mean that if $B_1   \dots   B_k=b_1\dots b_m$, then $J_1, \dots , J_k$ are blocks in cycles of $\b$, with $J_1 \dots  J_k=\a(b_1)\a(b_2) \dots  \a(b_m)$, and $|B_i|=|J_i|$, for $1\leq i \leq k$ (notice that there are $m$ possibilities for the first point $b_1$ in the first row). We refer to this notation as  the {\it block notation} (with respect to $\b$) of $\a|_{\set(\b_j)}$. If not required the subindex $k$ in  $\a|_{\set(\b_j),  k}$ it will be omitted. 
\begin{example} Let $\a, \b \in S_6$ with $\a=(1\;3\;4)(2\;5\;6)$ and $\b=(1 \; 2\;4\;5)(3\;6)$. If $\b_j=(1 \; 2\;4\;5)$, two ways to express $\a|_{\set(\b_j)}$ in block notation are
\[
\a|_{\set(\b_j), 3}=\left(
\begin{array}{|c|cc|c|ccccc}
1 & 2 & 4  & 5 \\
3 & 5 & 1  & 6
 \end{array}
\right), \hspace{0.3cm}
\a|_{\set(\b_j), 4}=\left(
\begin{array}{|c|c|c|c|ccccc}
1 & 2 & 4  & 5 \\
3 & 5 & 1  & 6
 \end{array}
\right),
\]
where the vertical lines denotes the limits of the blocks.
\end{example}
\section{Permutations that $k$-commute with a cycle of a permutation}\label{perthatkcommute}
In this section we show the relation between blocks in cycles of $\b$ and the permutations $\a$ that $k$-commutes with $\b$. First we prove the following.  
\begin{proposition}\label{12}
Let $\b$ be any permutation of cycle type $(c_1, \dots , c_n)$. Then $c(0, \b)=\prod_{i=1}^{n}i^{c_i}c_i!$, and $c(1,\b)=c(2,\b)=0$.  
\end{proposition}
\begin{proof}
When $k=0$, $c(0,\b)$ is the size of the centralizer of $\b$. As no two permutations have Hamming metric equal to $1$ then $c(1, \b)=0$. Now we show case $k=2$. It is easy to see that $H(\pi, \tau)=2$ if and only if $\pi\tau^{-1}$ is a transposition. If $H(\alpha \beta, \beta \alpha) = 2$ then $\alpha \beta\alpha^{-1}\beta^{-1}$ should be a transposition, but this lead to a contradiction because $\alpha \beta \alpha^{-1}\beta^{-1} $ is an even permutation.  
\end{proof}

If two permutations $\a$ and $\b$ commute, then $\a \b \a^{-1}=\b$, and we can think that ``$\alpha$ rearranges the cycles of $\b$". We will say that $\a$ {\it transforms} the cycle $\b_j$ of $\b$ into the cycle $\a \b_j\a^{-1}$ if $\a \b_j\a^{-1}$ is also a cycle in the cycle decomposition of $\b$. Let $B=b_1 \dots  b_l$ be a block in   $\b_j$, we say that $\a$ {\it commutes} with $\b$ {\it on the block} $B$ if $\a\b(b_i)=\b\a(b_i)$, for every $i=1, \dots, l$, and that {\it commutes} (resp. do not commute) with $\b$ on $\b_j$, or simply (by abusing of notation) that $\a$ commutes (resp. do not commute) with $\b_j$, if $\a\b(b)=\b\a(b)$ for every $b \in \set(\b_j)$ (resp. $\a\b(b) \neq \b\a(b)$ for some $b \in \set(\b_j)$). The following remark is implicitly used in some of the proofs in this article.
\begin{remark}\label{betaonblock}
Let $\a, \b \in S_n$ and $\{a_1,\dots, a_l\} \subseteq [n]$. If  $\a(a_1) \dots  \a(a_l)$ is a block in a cycle of $\b$ then, by definition of block, it follows that $\b(\a(a_i))=\a(a_{i+1})$,  $1\leq i \leq l-1$.
\end{remark}

The following result is the key to relate commutation and blocks in cycles. 
\begin{proposition}\label{pot}
Let $\a, \b \in S_n$. Let $\ell, m$ be integers, $1 \leq \ell < m \leq n$. Let $\b_j=(b_1 \; \dots  \; b_m)$ be a cycle of $\b$. If $\a$ commutes with $\b$ on the block $b_1 \dots  b_\ell$, then $\a(b_1)\dots  \a(b_\ell)\a(b_{\ell+1})$ is a block in a cycle of $\b$.
\end{proposition}
\begin{proof}
It is enough to prove that $\a(b_{i})=\b^{i-1}(\a(b_1))$, for $i=1, \dots , \ell+1$. The proof is by induction on $i \leq \ell+1$. The base case $i=1$ is trivial.  Assume as inductive hypothesis that the statement is true for every $k <\ell+1$. As $\a$ and $\b$ commute on $b_k$, we have that $\a(b_{k+1})=\a(\b(b_{k}))=\b(\a(b_{k}))=\b(\b^{k-1}(\a(b_1)))=\b^{k}(\a(b_1))$.
\end{proof}

We have the following result.
\begin{proposition}\label{transform}
Let $\b_j$ be an $m$-cycle of $\b$. Then  $\a$ commutes with $\b$ on $\b_j$ if and only if $\a$ transforms $\b_j$ into an $m$-cycle of $\b$. 
\end{proposition}
\begin{proof}
If $\a$ commutes with $\b$ on $\b_j=(b_1 \dots  b_m)$ then from Proposition~\ref{pot} it follows that $\a(b_1) \dots  \a(b_{m-1})\a(b_{m})$ is a block in a cycle  of $\b$. Now, as $\b(\a(b_{m}))=\a(\b(b_{m}))=\a(b_1)$ then $\a(b_1) \dots  \a(b_{m-1})\a(b_{m})$ is an improper block in an $m$-cycle, say $\b_l$, of $\b$, i.e., $\a$ transforms $\b_j$ into $\b_l$. Conversely, if $\a$ transform $\b_j$ into an $m$-cycle, say $\b_l$, of $\b$ then $\b_l=\a \b_j \a^{-1}$ that is equal to $(\a(b_1) \dots  \a(b_{m}))$. Then $\a(b_1) \dots   \a(b_{m})$ is a block in a cycle of $\b$ which implies (see Remark~\ref{betaonblock}) that $\b(\a(b_i))=\a(b_{i+1 \bmod m})$. As also $\a(b_{i+1 \bmod m})=\a(\b(b_i))$ then $\b(\a(b_i))=\a(\b(b_i))$, for every $i \in \{1, \dots , m\}$.  
\end{proof}
\begin{corollary}
Let $\a, \b \in S_n$. Then $\a$ and $\b$ commute if and only if $\a$ transforms all the cycles of $\b$ into cycles of $\b$. 
\end{corollary}
\begin{remark} \label{blockcommute}
Using block notation, Proposition~\ref{transform} can be rewritten as follows: a permutation $\a$ commutes with $\b$ on $\b_j=(B)$  if and only if
\[
\a|_{\set(\b_j), 1}=\left(\begin{array}{cccc}
B  \\
 B' 
 \end{array} \right),
\]  
where $\b_{l}=(B')$ is a cycle of $\b$.
\end{remark}

 Let $\b_i$ be a cycle of $\b$. We say that $\a$ {\it $(k_i, \b)$-commutes} with $\b_i$ if there exists exactly $k_i$ points in $\b_i$ on which $\a$ and $\b$ do not commute. 
\begin{remark}
Let $\beta$ be a permutation. Then $\alpha$ $k$-commutes with $\beta$ if and only if there exists $h$ cycles, say $\b_1,  \dots , \b_h$, of $\beta$ such that for every $1  \leq  i  \leq \ h$, $\alpha$ $(k_i, \b)$-commutes with $\b_i$, $k_i \geq 1$, where  $k_1+\dots +k_h = k$ and  $\alpha$ commutes with $\b$ on every cycle not in $\lbrace \b_1, \dots , \b_h \rbrace$.
\end{remark}

Now we present one of our main results.

\begin{theorem}\label{1cycleblocks}
Let $\b_j$ be an $m$-cycle of $\b$ and $k\geq1$. Then $\a$ $(k, \b)$-commutes with $\b_j$ if and only if $\a \b_j\a\inv=(P_1\dots  P_k)$,  where the blocks $P_1, \dots , P_k$ satisfy the following
\begin{enumerate}
\item if $k=1$ then $P_1$ is a proper block in a cycle of $\b$,
\item if $k >1$ then $P_1, \dots , P_k$  are $k$ pairwise disjoint blocks, from one or more cycles of $\b$, such that for any $i \in [k]$, the string $P_iP_{i+1\bmod k }$ is not a block in any cycle of $\b$. 
\end{enumerate}
\end{theorem}
\begin{proof}
(1)  $( \Leftarrow)$: Suppose that $\a\b_j\a^{-1}=(P_1)$ and without lost of generality we can take $\b_j=(b_1  \dots   b_m)$ such that  $P_1=\a(b_1) \dots  \a(b_m)$. As, by hypothesis, $P_1$ is a proper block in a cycle of $\b$ then $\b(\a(b_i))=\a(b_{i+1})$, for $i=1, \dots, m-1$ (Remark~\ref{betaonblock}). As $\a(b_{i+1})=\a(\b(b_i))$, we have that $\a$ and $\b$ commute on $b_i$, for every $1\leq i \leq m-1$. Now we prove that $\b(\a(b_m)) \neq \a(\b(b_m))$ by contradiction. Suppose that $\b(\a(b_m)) = \a(\b(b_m))$,  then $\b(\a(b_m))=\a(b_1)$ which implies that $P_1$ is an improper block in a cycle of $\b$, a contradiction. 
 
$( \Rightarrow)$: Without lost of generality we can assume that $\a$ commutes with $\b$ on block $b_1 \dots   b_{m-1}$ of $\b_j=(b_1  \dots b_{m-1}b_m)$ and that does not commute on $b_m$. By Proposition~\ref{pot}, $\a(b_1) \dots  \a(b_{m-1})\a(b_m)$ is a block in a cycle of $\b$ and is a proper block due to Proposition~\ref{transform} (see also Remark~\ref{blockcommute}).

(2) 
$( \Leftarrow)$:  Suppose that $\a \b_j \a\inv=(P_1 \dots  P_k)$, where every $P_i=p_{i1}p_{i2} \dots  p_{i\ell_i}$ is a block in a cycle of $\b$ and  that for every $i \in [k]$, $P_iP_{i+1}$ is not a block in any cycle of $\b$. Without lost of generality we assume that $\b_j=(b_{11} \dots b_{1\ell_1} b_{21} \dots b_{2\ell_2}  \dots  b_{k1} \dots    b_{k\ell_k})$ such that $\a(b_{ir})=p_{ir}$, for every $i \in [k]$ and $1\leq r \leq \ell_i$.
 
As $P_i$ is a block in a cycle of $\b$ then $\b(p_{ir})=p_{i(r+1)}$, for every $1\leq r < \ell_i$ (Remark~\ref{betaonblock}), then we have for one side that $p_{i(r+1)}=\b(p_{ir})=\b\big(\a(b_{ir})\big)$ and for the other side $p_{i(r+1)}=\a(b_{i(r+1)})=\a\big(\b(b_{ir})\big)$ and hence $\a$ and $\b$ commute on $b_{ir}$, for every $1\leq r < \ell_i$. Now we prove that $\b\big(\a(b_{i\ell_i})\big)\neq \a\big(\b(b_{i\ell_i})\big)$ by contradiction. Suppose that $\b\big(\a(b_{i\ell_i})\big)=\a\big(\b(b_{i\ell_i})\big)$, as  
$\b(p_{i\ell_i})=\b\big(\a(b_{i\ell_i})\big)$, then $\b(p_{i\ell_i})=\a\big(\b(b_{i\ell_i})\big)=\a(b_{(i+1 \bmod k)1})=p_{(i+1 \bmod k)1}$, which implies that  $P_iP_{i+1 \bmod k}$ is a block in a cycle of $\b$, a contradiction.

$( \Rightarrow)$: If $\a$ does not commute with $\b$  on exactly $k$ points in $\b_j$ we write $\b_j=(B_1 \dots  B_k)$, where for every block $B_i=b_{i1}b_{i2} \dots  b_{i\ell_i}$, $\a$ and $\b$ commute on $b_{ij}$, for $1\leq j <\ell_i$, and does not commute on $b_{i\ell_i}$.  By Proposition~\ref{pot} we have that $P_i:=\a(b_{i1})\a(b_{i2}) \dots  \a(b_{i\ell_i})$ is a block in a cycle of $\b$. Now suppose that for some $i$, $P_iP_{i+1 \bmod k}$ is a block in any cycle of $\b$, then $\b\big(\a(b_{i\ell_i})\big)=\a(b_{(i+1 \bmod k)1})=\a\big(\b(b_{i\ell_i})\big)$,  contradicting the assumption that $\a$ and $\b$ do not commute on $b_{i\ell_i}$.

\end{proof}

\begin{remark}\label{kcommuteblock}
$\a|_{\set(\b_j), k}$ in previous proposition, can be written as
\[
\a|_{\set(\b_j), k}=\left(\begin{array}{cccc}
B_1  \dots   B_k   \\
 P_1    \dots   P_k  
 \end{array} \right),
\]  
with $\b_j=(B_1  \dots   B_k)$, $\a\b_j\a^{-1}=(P_1 \dots  P_k)$, where for every $1\leq i \leq k$, $|B_i|=|P_i|=\ell_i$, $P_i=p_{i1}p_{i2} \dots  p_{i\ell_i}$, $B_i=b_{i1}b_{i2} \dots  b_{i\ell_i}$, i.e., $\a(b_{ir})=p_{ir}$, $1\leq r\leq \ell_i$, and where $\a$ and $\b$ commute on   $b_{i1}b_{i2} \dots  b_{i(\ell_i-1)}$ and do not commute on $b_{i\ell_i}=\a^{-1}(p_{i\ell_i})$ (the last point of $B_i$).  
\end{remark}

Using Theorem~\ref{1cycleblocks} we can characterize permutations that $k$-commute with $\b$ in terms of blocks in cycles of $\b$ as follows.
\begin{corollary}
Let $\a, \b \in S_n$. Then $\a$ $k$-commutes with $\b$ if and only if there exist $h$ cycles of $\b$, say $\b_1, \dots , \b_h$, such that $\a$ commutes with $\b$ on each cycle not in $\{\b_1, \dots, \b_h\}$ and for every $i \in \{1, \dots , h\}$, 
$\a \b_i\a\inv=(P_1^{(i)}\dots  P^{(i)}_{k_i})$, with $k_i\geq 1$,  $k=k_1+\dots  +k_h$, and where the blocks $P_1^{(i)},  \dots , P^{(i)}_{k_i}$ satisfy the following
\begin{enumerate}
\item if $k_i=1$ then $P^{(i)}_1$ is a proper block in a cycle of $\b$,
\item if $k_i >1$ then $P^{(i)}_1, \dots , P^{(i)}_{k_i}$  are $k_i$ pairwise disjoint blocks, from one or more cycles of $\b$, such that for any $r \in [k_i]$, $P^{(i)}_rP^{(i)}_{r+1\bmod k_i }$ is not a block in any cycle of $\b$,
\item  $\{P_1^{(1)}, \dots, P_{k_1}^{(1)}, \dots ,P_1^{(h)}, \dots , P_{k_h}^{(h)}\}$ is a set of pairwise disjoints blocks from one or more cycles of $\b$.
\end{enumerate}
 \end{corollary}
 
\begin{example}
Let $\a, \b \in S_7$, with $\b=(1 \;2\;4 \; 5\;3)(7\;6) $ and $\a=(2\;7)(3\;6 \;4 \;5)$. By direct calculations we can check that $\a$ $(4, \b)$-commutes with $\b_1=(1 \;2\;4 \; 5\;3)$ (the b.c.p. are $1, 2, 3$ and $5$) and $(1, \b)$-commutes with $\b_2=(7\;6)$ (the b.c.p. is $6$). In block notation $\a|_{\set(\b_1)}$ and $\a|_{\set(\b_2)}$ can be expressed as 
\[
\a|_{\set(\b_1), 4}=\left(\begin{array}{|c|c|cc|c|}
1 &  2 & 4 & 5 & 3   \\
1 & 7  & 5  & 3  &6
 \end{array}\right),
  \hspace{0.3cm}
  \a|_{\set(\b_2), 1}=\left(\begin{array}{|cc|}
7 &  6 \\
2 & 4
 \end{array}\right).
 \]
\end{example}
As a first application of Theorem~\ref{1cycleblocks} we present the following generalization of  Lemma 2(b) in \cite{gor} where it was proved only for the case when $\b$ is a product of $m$ disjoint $l$-cycles. 

\begin{proposition}\label{corone}
Let $\b$ be any $n$-permutation whose maximum cycle length in its cycle decomposition is $m$.  If $\a$ commutes with $\b$ on $m-1$ points in an $m$-cycle $\b_j$ of $\b$ then $\a$ commutes with $\b$ on $\b_j$.   
\end{proposition}
\begin{proof}
Suppose that $\a$ and $\b$ do not commute on the remaining point in $\b_j$. By part (1) of Theorem~\ref{1cycleblocks}, $\a \b_j \a^{-1}=(P)$, where $P$ is a proper block in an $l$-cycle of $\b$, i.e., $l > m$, but this is a contradiction because $m$ is the maximum cycle length of cycles in $\b$.  
\end{proof}
The following proposition will be useful in the proofs of some of our results.
\begin{proposition}\label{numberbc}
Let $\a$ and $\b$ be two permutations that $k$-commute, $k>0$. Suppose that $\a$ does not commute with $\b$ on the cycles $\b_1, \dots, \b_r$, of lengths $l_1, \dots l_r$, respectively, and that commutes with the rest of the cycles of $\b$ (if any). Then there exists exactly $r$ cycles, say $\b_1', \dots, \b_r'$, of lengths $l_1, \dots l_r$, respectively, such that $\a\left(\set(\b_1) \cup \dots \cup \set(\b_r)\right)=\set(\b_1') \cup \dots \cup \set(\b_r')$. Even more, suppose that $\a$ does not commute with exactly $h_i$ $i$-cycles of $\b$ and that commutes with the rest of the $i$-cycles of $\b$ (if any). Then there exists exactly $h_i$ $i$-cycles of $\b$ such that each of them contains at least one point that is the image under $\a$ of one b.c.p. of $\a$ and $\b$.
\end{proposition} 
\begin{proof}
 Let $\b_{r+1}, \dots, \b_{s}$ the rest of cycles of $\b$ of lengths $l_{r+1}, \dots l_s$, respectively. As $\a$ commutes with $\b$ on every one of this cycles then $\a$ transforms each $\b_{t}$ into an $l_t$-cycle $\b_t'$, for $r+1 \leq t \leq s$ and then there exists $\b_{r+1}', \dots, \b_{s}'$ cycles of $\b$ of lengths $l_{r+1}, \dots, l_s$, respectively such that $\a\left(\set(\b_{r+1}) \cup \dots \cup \set(\b_s)\right)=\set(\b_{r+1}') \cup \dots \cup \set(\b_s')$ and the result of the first part of the proposition follows because $\a$ is a bijection. The previous argument also implies that if $\a$ does not commute with exactly $h_i$ $i$-cycles of $\b$ and that commutes with the rest of the $i$-cycles of $\b$ (if any) then there exists exactly $h_i$ $i$-cycles of $\b$, say $\b_1'', \dots, \b_{h_i}''$, such that $\b_t''\neq \a \b_j \a^{-1}$ for any cycle $\b_j$ of $\b$. The following claim completes the proof of the second part 
  \begin{claim}\label{c-numberbc} If all the points in an $i$-cycle $\b_{j'}$ of $\b$ are images under $\a$ of g.c.p., then $\b_{j'}=\a \b_j \a^{-1}$ for some $i$-cycle $\b_j$ of $\b$.
 \end{claim} 
 \begin{proof}
 We first prove that if all the points in the $i$-cycle $\b_{j'}$ are images under $\a$ of $\b$ of g.c.p. of $\a$ and $\b$ then these g.c.p. belong to exactly one $l$-cycle $\b_j=(b_1 \dots  b_l)$. If $\b_{j'}$ contains the images under $\a$ of g.c.p. in different cycles of $\b$, then $\b_{j'}$ contains the string $\a(x)\a(y)$ with $x$ and $y$ in different cycles of $\b$, i.e., $\b(x) \neq y$, but this implies that $x$ is a b.c.p. because $\a(\b(x))\neq \a(y)= \b(\a(x))$. Now we show that $l \leq i$. Suppose that $l >i$ then, and without lost of generality, we have that $\b_j'=(\a(b_1)\dots \a(b_i))$, i.e., $\b_j'(\a(b_t))=\a(b_{t+1 \mod i})$, for every $t$ (if $\b_j'(\a(b_t)) \neq \a(b_{t+1 \mod i})$ for some $t$ then $b_t$ will be a b.c.p.). But this implies that $b_i$ is a b.c.p. because $\b(\a(b_i))=\a(b_1) \neq \a(b_{i+1})=\a(\b(b_i))$, which is a contradiction. Therefore $l \leq i$, and this will implies that $l=i$ (by using a similar argument as the previous one), i.e. $\b_{j'}= \a\b_j\a^{-1}$.
  \end{proof} 
\end{proof}
\subsection{Permutations that $(k, \b)$-commute with a cycle of $\b$}

Let $k \geq 3$ be a positive integer. Let  $\a$ be any permutation that $k$-commutes with $\b$ and that $(k, \b)$-commutes with an $m$-cycle, say $\b_j$, of $\b$, i.e., all the b.c.p. of $\a$ and $\b$ are in $\b_j$. From Proposition~\ref{numberbc} it follows that there exist exactly one  $m$-cycle, say $\b_{j'}$, of $\b$ such that $\set(\b_{j'})=\a(\set(\b_j))$. Using this fact we present a procedure (Algorithm 1) that allows us to obtain any such permutation $\a$. First we give some definitions. 
The {\it canonical cycle notation} of a permutation  $\pi$ is defined as follows: first, write the largest element of each cycle, and then arrange the cycles in increasing order of their first elements. Let $\pi$ be a permutation written in its canonical cycle notation, the {\it transition function} of $\pi$ from canonical cycle notation to one-line notation is the map $\Psi: S_n \rightarrow S_n$ that sends $\pi$ to the permutation $\Psi(\pi)$ written in one-line notation that is obtained from $\pi$ by omitting all the parentheses. This map is a bijection (see, e.g., \cite[p. 96]{bona2}). 

\begin{example}
Let $\pi \in S_7$ be $(431)(65)(72)$ ($\pi$ is written in its canonical cycle notation). Then  $\Psi(\pi)=4316572$.
\end{example}
\begin{algorithm}

\begin{description}
\item[]
\item[Step 1.] Choose $m$-cycles $\b_{j}$ and $\b_{j'}$ of $\b$  (with the possibility that $\b_{j'}=\b_j$), and with $m \geq k$.
 \item[Step 2]  Choose a subset of $k$ points of $\set(\b_{j'})$. Let $P=p_1 p_2    \dots  p_m$ be any improper block of $\b_{j'}$ such that $p_m$ is one of the selected points. Let $p_{h_1}, p_{h_2}, \dots, p_{h_k}=p_m$ the $k$ selected points whit $h_1 < \dots  < h_k$. Now partition the improper block $P$ into $k$ blocks, $P_1, P_2,   \dots ,  P_k$, as follows 
 \[
 \underbrace{p_1 \dots  p_{h_1}}_{P_1}\underbrace{p_{h_1+1} \dots  p_{h_2}}_{P_2} \dots  \underbrace{p_{h_{k-1}+1} \dots  p_{h_k}}_{P_k}, 
 \]
i.e., $P=P_1P_2  \dots   P_k$, and where $p_{h_r}$ is the last point of $P_r$, $1\leq r \leq k$.  
 \item[Step 3]    Choose any $k$-cycle $\tau$ of $[k]=\{1, \dots ,k\}$ with $\tau(a) \neq a+1 \bmod k$, for every $a\in [k]$, and make the block permutation 
 \[
P':=P_{\Psi(\tau)(1)} P_{\Psi(\tau)(2)} \dots  P_{\Psi(\tau)(k)}=P_{i_1} P_{i_2} \dots  P_{i_k},
\]
where $\Psi(\tau)$ is the transition function of the canonical cycle notation of $\tau$ to one-line notation (notice that $\tau$ in its canonical cycle notation is equal to $(i_1 \dots  i_k)$).
\item[Step 4]  Construct $\a|_{\set(\b_j)}: \set(\b_j) \rightarrow \set(\b_{j'})$ in block notation as it follows: 
\[
\a|_{\set(\b_j), k}=\left(
\begin{array}{ccccccccccccccc}
 B_1& B_2& \dots  & B_k \\
P_{i_1}         &  P_{i_2}                & \dots  &   P_{i_k}             
 \end{array}
\right).
\]
where $\b_j=(B_1\dots  B_k)$ and  $|B_r|=|P_{i_r}|$, $1\leq r \leq k$.
\item[Step 5.] Construct $\a|_{[n] \setminus \set(\b_j)}: [n] \setminus \set(\b_j) \rightarrow [n] \setminus \set(\b_{j'})$ as any bijection that commutes with $\b|_{[n] \setminus \set(\b_j)}: [n] \setminus \set(\b_j) \rightarrow [n] \setminus \set(\b_{j'})$. 

\end{description} 
\end{algorithm}
Notice that Step 5 is possible because $\a$ can be constructed in such a way that it transforms the $c_m-1$ $m$-cycles of $\b$ different than $\b_j$ (if any) into the $c_m-1$ $m$-cycles of $\b$ different than $\b_{j'}$ (if any),  and that transforms the $l$-cycles of $\b$ (if any), with $l \neq m$, into $l$-cycles of $\b$ (if any). The following two propositions shows that Algorithm 1 produces all the permutation $\a$ with the desired properties.
\begin{proposition}\label{algorithmworks}
Permutation $\a$ constructed with Algorithm 1 does not commute with $\b$ on all points in $\mathcal{A}:=\a^{-1}(\{p_{h_1},\dots , p_{h_k}\})$ and commutes with $\b$ on all points in $[n] \setminus \mathcal{A}$. 
\end{proposition}
\begin{proof}
Let $\b_j$ and $\b_{j'}$ be the cycles of $\b$ selected in Step 1 of Algorithm 1, and $\{p_{h_1}, \dots , p_{h_k}\} $ the subset of $\set(\b_{j'})$ selected in Step 2. By the way in which $\a$ is constructed of in Step 3 and 4, $\a \b_j \a^{-1}$ has $P_{i_1} P_{i_2}                \dots    P_{i_k}$ as an improper block, where $P_{i_r}P_{i_{r+1 \bmod k}}$, $1\leq r \leq k$, is not a block in any cycle of $\b$ (by Step 3, $i_{r+1 \bmod k} - i_r \bmod k \neq 1$). From Theorem~\ref{1cycleblocks}, we have that $\a$ does not commute with $\b$ on exactly $k$ points in $\set(\b_j)$. Even more, in the proof of Theorem~\ref{1cycleblocks} was showed that $\a$ and $\b$ do not commute on $\a^{-1}(p_{h_r})$, $1\leq r \leq k$ (see Remark \ref{kcommuteblock}).  Finally, by the  construction of $\a$ in Step 5, $\a$ and $\b$ commute on all points in  $[n] \setminus \set(\b_j)$. 
\end{proof}
\begin{proposition}\label{clama}  Let $k \geq 3$. Any permutation $\a$ that $k$-commutes with $\b$ and such that all the b.c.p. of $\a$ and $\b$ are in one cycle of $\b$ can be obtained with  Algorithm 1.
\end{proposition}
\begin{proof}
Let $\a$ be any permutation that $k$-commutes with $\b$ and such that all the b.c.p. are in exactly one $m$-cycle, say $\b_j$, of $\b$. From Proposition~\ref{numberbc} it follows that there exists exactly one $m$-cycle, say  $\b_{j'}$, of $\b$ such that $\a(\set(\b_j))=\set(\b_{j'})$. By Theorem~\ref{1cycleblocks}, we have that $\a \b_j \a^{-1}=(P_1 \dots  P_k)$, where $P_{1}, \dots , P_{k}$ are $k$ pairwise disjoint blocks in $\b_{j'}$ and $P_{r} P_{r+1 \mod k}$ is not a block in any cycle of $\b$, $1\leq r \leq k$. As $\a(\set(\b_j))=\set(\b_{j'})$,  we have that $P_{1} \dots  P_{k} $  is a block permutation of $B'=P_{i_1} \dots  P_{i_k}$, where $\b_{j'}=(B')$.  Now, rename the blocks $P_{i_s}$ as $B'_s$ to obtain $B'=B'_1 \dots  B'_k$. In this way, $\a \b_j \a^{-1}=(B'_{l_1}\dots  B'_{l_k})$ with  $l_{r+1 \bmod k}-l_r \bmod k \neq 1$, $1\leq r \leq k$. Indeed, if $l_{r+1 \bmod k}-l_r \bmod k= 1$ for some $r \in \{1, \dots , k\}$, then $B'_{l_r}B'_{l_{r+1 \bmod k}}$ will be a block in $\b_{j'}$, and hence the number of b.c.p. of $\a$ and $\b$ will be less than $k$, which is a contradiction. 

As $\a \b_j \a^{-1}=(B'_{l_1} \dots  B'_{l_k})=(B'_{l_2} \dots  B'_{l_1})=\dots  =(B'_{l_k} \dots  B'_{l_{k-1}})$, we can assume without lost of generality that $l_1=k$ (from these $k$ expressions, choose the one that begins with block $B_k'$).
Then $\a|_{\set(\b_j)}$ can be written as 

\[
\a|_{\set(\b_j), k}=\left(
\begin{array}{ccccccccccccccc}
 B_1&  \dots  & B_k \\
B'_{l_1}         &  \dots  &   B'_{l_k}              
 \end{array}
\right), 
\]
where $\b_j=(B_1\dots  B_k)$, and $|B_i|=|B'_{h_i}|$, $1\leq i \leq k$.

Now, we consider $l_1 l_2 \dots  l_{k}$ as a permutation, named $\pi$, of $\{1, \dots , k\}$ in one-line notation. As $l_1$ (that is equal to $k$) is the greatest element in $\{l_1, \dots , l_k\}$, then $\tau :=\Psi^{-1}(\pi)=(l_1 \dots  l_{k})$, where $\Psi$ is the transition function of the canonical cycle notation to one-line notation.  Notice that $\tau$ is a $k$-cycle  in $S_k$ such that $\tau(a) \neq a+1$, for any $a \in [k]$. Thus we conclude that $\a|_{\set(\b_j)}$ can be obtained by Steps 1 to 4 of Algorithm 1. Now as $\a$ commutes with $\b$ on all different cycles from $\b_j$, $\a|_{[n]\setminus \set(\b_j)}$ can be obtained with Step 5 of Algorithm 1.
 \end{proof}
\section{On the number $c(k, \b)$}\label{numberckb}
 In this section we present some results about the number $c(k, \b)$. First we show that for any nonnegative integer $k$ and any $\b \in S_n$, the number $c(k, \b)$ is invariant under conjugation. Let $C(k, \b)=\{\a \in S_n : H(\a\b, \b\a)=k\}$ (so that $c(k, \b)=|C(k, \b)|$)
  
\begin{proposition}\label{propeqcar} Let $\b \in S_n$. Then $c(k, \tau \b \tau^{-1}) =c(k, \b)$ for any $\tau \in S_n$.

\end{proposition}
\begin{proof}({\it Sketch})
For $\tau \in S_n$, let $\tau C(k, \b) \tau^{-1}:=\{\tau \a \tau^{-1} : \a \in C(k, \b)\}$. By the bi-invariance of the Hamming metric is straightforward to show that $C(k, \tau \b \tau^{-1})=\tau C(k, \b) \tau^{-1}$.   
Now, it is easy to check that the function $\phi: C(k, \b) \rightarrow  \tau C(k, \b) \tau^{-1}$ given by $\sigma \mapsto \tau \sigma \tau^{-1}$ is a bijection, so we have that  $ |C(k, \b)|=|\tau C(k, \b) \tau^{-1}|= | C(k, \tau \b \tau^{-1})|$.

\end{proof}

\begin{remark} 
If $\b$ and $\b'$ are conjugate permutations it is not always true that $C(k, \b)=C(k, \b')$. 
For example, let  $\b=(12345) \in S_5$ and $\b'=(23145)$. If $\a=(14)(25)$ then $H(\a\b, \b\a)=3$ and $H(\a\b', \b'\a)=5$, which implies that $C(3, \b) \neq C(3, \b')$.
\end{remark}

The following result shows that $c(k, \b)$ is a multiple of  $|C_{S_n}(\b)|$.  
\begin{proposition}\label{union} Let  $\b \in S_n$. Suppose that $C(k,\b)$ is a non-empty set. Then 
\[
C(k,\b)=\bigcup_{\a \in C(k, \b)}C_{S_n}(\b)\a.
\]
\end{proposition}
\begin{proof}
The inclusion $C(k,\b) \subseteq \bigcup_{\a \in C(k, \b)}C_{S_n}(\b)\a$ is clear. Now, let  $\rho \in \bigcup_{\a \in C(k, \b)}C_{S_n}(\b)\a$,   then $\rho= \tau \a$ for some $\tau \in C_{S_n}(\b)$ and some $\a \in C(k, \b)$. So we have that
$
H(\rho \b,\b \rho)=H(\tau \a \b,\b \tau \a)=H(\tau \a \b,\tau \b \a)=H(\a \b,\b \a)=k, 
$
and then $\rho \in C(k, \b)$.
\end{proof}
\begin{corollary}
 Let  $\b \in S_n$ and let $k$ be any non-negative integer. Then $c(k,\b)$ is a multiple of $|C_{S_n}(\b)|$.
\end{corollary}

\subsection{Number of even permutations in $C(k, \b)$}

One interesting application of Proposition~\ref{union} is that for $\b$ of some special cycle type we can find the proportion of even permutations in $C(k, \b)$. We say that a permutation is a  {\it cdoi-permutation} if  its  cycle type consist of distinct odd integers. We need the following propositions.

\begin{proposition}\label{not-c-odd}\cite[exercise 21, p.\ 131.]{dum}
The permutation $\s \in S_n$ does not commute with any odd permutation if and only if the cycle type of $\s$ consist of distinct odd integers.
\end{proposition}
\begin{proposition}\label{c-con-odd}\cite[exercise 3.22, p.\ 51.]{rot}
If $G \leq S_n$ contains an odd permutation, then $|G|$ is even, and exactly half of the elements in $G$ are odd permutations.
\end{proposition}
We are ready to prove the following.
\begin{proposition}
Let $\b \in S_n$ be a permutation that is not a cdoi-permutation. Then exactly one half of the permutations in $C(k, \b)$ are odd. 
\end{proposition}
\begin{proof}
As the cycle type of $\b$ does not consist only of distinct odd integers, then from Proposition \ref{not-c-odd} it follows that $\b$ commutes with an odd permutation, i.e., $C_{S_n}(\b)$ contains at least one odd permutation. By Proposition \ref{c-con-odd} we have that exactly one half of the elements in $C_{S_n}(\b)$  are odd permutations. Then, for any $\a$ in $C(k, \b)$, the number of odd permutations in $\cen(\b)\a$ remains one half. Finally,  as  
$
C(k,\b)=\bigcup_{\a \in C(k, \b)}C_{S_n}(\b)\a 
$
(by Proposition \ref{union}), we have that exactly one half permutations in $C(k, \b)$ are odd.
\end{proof}

\begin{corollary}
If $\b$ is not a cdoi-permutation then $c(k, \b)$ is even.
\end{corollary}

\begin{question} Let $\b$ be a cdoi-permutation, what is the number of even permutations that $k$-commute with $\b$?
\end{question}

\subsection{On the number $c(\lam_{k^{(1)}}, \b)$}

We use the following notation: Let $\lam_k=[k_1, \dots , k_h]$ denotes an integer partition of $k$, with $k_i \geq 1$. 
If $\lam_k=[k_1, \dots , k_h]$ is an integer partition of $k$, we use  $C(\lambda_k, \b)$, or $C([k_1, \dots , k_h], \b)$, to denote the set of permutations that $k$-commutes with $\b$ such that for every $\alpha \in C([k_1, \dots , k_h], \b)$ there exists exactly $h$ cycles, says $\b_1, \dots, \b_h$, in the cycle decomposition of $\b$, where $\alpha$ $(k_1, \b)$-commutes with $\b_1$, $(k_2, \b)$-commutes with $\b_2$, \dots , $(k_h, \b)$-commutes with $\b_h$. We will use $c(\lambda_k, \b)$ or $c([k_1, \dots , k_h], \b)$ to denote the cardinality of $C(\lambda_k, \b)$. We will use the notation $c(\lam_{k^{(1)}}, \b)$ and $c(\lam_{k^{(k)}}, \b)$ instead of $c([k_1], \b)$ and $c([1, \dots , 1], \b)$, respectively.

\begin{example}
Let $\a, \b \in S_{14}$, with $\b=(1\;2)(3\;4\;5)(6\;7\;8)(9\;10\;11\;12)(13\;14)$ and $\a=(1\; 3\; 9\;6)(2\;4\;10\;7)(5\;11\;8)$. We have that $\a$ $6$-commutes with $\b$ in the following way: $\a$ $(1, \b)$-commutes with $(1\;2)$, $(1, \b)$-commutes with $(3\;4\;5)$, $(2, \b)$-commutes with $(6\;7\;8)$, $(2, \b)$-commutes with $(9\;10\;11\;12)$ and commutes with $\b$ on $(13\;14)$. 
\end{example}

\begin{remark}
We are considering unordered partitions, for example $[2, 2, 1]=[2, 1, 2]$.
\end{remark}

Let  $L(\b)$ denotes the set of cycle lengths in the cycle decomposition of $\b$ including $1$-cycles. Let $\mc_l(\b)$ be the set of all $l$-cycles in the cycle decomposition of $\b$. Let $\ms(\b)$ be a subset of the set of all the cycles in $\b$. Suppose that we want to construct a permutation $\alpha$ that will not commute with $\b$ on all the cycles in $\ms(\b)$ (for every cycle $\b_j \in \ms(\b)$, $\alpha$ and $\b$ will not commute on at least one point in $\b_j$) and that commutes on every point of the remaining (if any) cycles of $\b$.  Let $L_\ms$ be the set of lengths of cycles in $\ms(\b)$ and let $\ms_l$ be the set of all cycles of length $l$ in $\ms(\b)$. From Proposition~\ref{numberbc} it follows that we can construct $\alpha$ by obtaining two bijections $\alpha|_R: R \rightarrow R'$ and $\alpha|_{\overline{R}}:\overline{R} \rightarrow \overline{R'}$, where $R= \bigcup_{\b_j \in \ms(\b)} \set(\b_j)$, $\overline{R}=[n] \setminus R$,  with the following steps ($R'$ and $\overline{R'}$ are defined in Step 2 below).

\begin{algorithm}\label{algo2}

\begin{enumerate}
\item[]
\item[Step 1] For every $l \in L_\ms$, select a subset $\ms'_l(\b) \subseteq \mc_l(\b)$ of cardinality $h_l:=|\ms_l(\b)|$ and obtain $\ms'(\b):=\bigcup_{l \in 	L_\ms} \ms'_l(\b)$ (the cycles in $\ms'(\b)$ will contain the images under $\a$ of points in cycles in  $\ms(\b)$. 
\item[Step 2] Construct a bijection $\a|_R: R \rightarrow R'$, where $R'= \bigcup_{\b_j \in \ms'(\b)} \set(\b_j)$, in such a way that $\a|_R$ and $\b|_R$ will not commute  on the desired points. 
\item[Step 3] Construct a bijection $\a|_{\overline{R}}: \overline{R} \rightarrow \overline{R'}$, where $\overline{R'}=[n] \setminus R'$, as any bijection that commutes with $\b|_{\overline{R}}$.
\end{enumerate}
\end{algorithm}
The construction in Step 3 can be done as follows: for every $l \in L_\ms$ (resp. $l \not \in L_\ms$), construct $\a|_{\overline{R}}$ in such a way that $\a$ transforms cycles in $\mc_l(\b) \setminus \ms_l(\b)$ into cycles in $\mc_l(\b) \setminus \ms'_l(\b)$ (resp.  that transforms cycles in $\mc_l(\b)$ into cycles in $\mc_l(\b)$). The more difficult part is to obtain all bijections $\a|_R$ in Step 2. Let $r_0$ denotes the number of ways to construct $\a|_R$ in Step 2 of Algorithm~\ref{algo2}. The following technical proposition will be useful to simplify some of the calculations in some of the enumerative results in this paper (and also shows that the main enumerative problem is to find $r_0$).

\begin{proposition}\label{partecentral}
Let  $H=\{\ell_1, \dots, \ell_i\}$ be a subset of $L(\b)$. The number of permutations $\a$ that does not commute with $\b$  on exactly $h_{\ell_1}$ $\ell_1$-cycles of $\b$, $h_{\ell_2}$ $\ell_2$-cycles of $\b$, \dots, $h_{\ell_i}$ $\ell_i$-cycles of $\b$ and that commutes with the rest of cycles of $\b$ is equal to 
\begin{equation*}
r_0  |C_{S_n}(\b)| \prod_{\ell \in H} \frac{1}{h_\ell ! \ell^{h_\ell}}{c_\ell \choose h_\ell} .
\end{equation*}
\end{proposition}
\begin{proof}
We will count the number of ways to select the set $\ms(\b)$ of all cycles of $\b$ where $\alpha$ and $\b$ will not commute and the set $\ms'(\b)$ in Step 1 of Algorithm~\ref{algo2}. Let $c_\ell$ denotes the number of $\ell$-cycles in the cycle decomposition of $\b$. For every $\ell \in H$, there are ${c_\ell \choose h_\ell}$ ways to select the $\ell$-cycles that will belong to  $\ms(\b)$ and ${c_\ell \choose h_\ell}$ ways to select the $\ell$-cycles that will belong to $\ms'(\b)$. Now we enumerate all bijections  $\a|_{\overline{R}}$ that can be obtained as in Step 3 of Algorithm~\ref{algo2}. Once we have selected  $\ms(\b)$ and $\ms'(\b)$, for every $l \in H$, there are $(c_\ell-h_\ell)!l^{c_\ell-h_\ell}$ ways to make that $\a$ transforms the $c_\ell-h_\ell$ cycles of length $\ell$ of $\b$ not belonging to  $\ms(\b)$ into the $c_\ell-h_\ell$ cycles of length $\ell$ of $\b$ not belonging to  $\ms'(\b)$. For $l \not \in H$, there are $|C_{S_n}(\b)|\prod_{l \in H}\frac{1}{l^{c_\ell}c_\ell!}$ ways to make that $\a$ transforms the $c_\ell$ $\ell$-cycles of $\b$ into $c_\ell$ $\ell$-cycles of $\b$. Then we have that the number of permutations $\a$ that satisfies the desired conditions is equal to
\begin{equation*}
r_0 \prod_{l \in H}{c_\ell \choose h_\ell}^2\prod_{l \in H}  (c_\ell-h_\ell)!l^{c_\ell-h_\ell} \left( |C_{S_n}(\b)|\prod_{l \in H}\frac{1}{l^{c_\ell}c_\ell!}\right)=r_0 |C_{S_n}(\b)|\prod_{l \in H}\frac{ 1}{h_\ell!l^{h_\ell}}{c_\ell \choose h_\ell}.
\end{equation*}

\end{proof}
Let $f(k)$ be the number of cyclic permutations ($k$-cycles) of $\{1, \dots , k\}$ with no $i \mapsto i+1\bmod k$ (see, e.g., \cite[exercise 8, p.\ 88]{stanley} or sequence A000757 in OEIS \cite{oeis}). 
 \begin{theorem} \label{onlymcycles}
Let $\b \in S_n$ be of type $(c_1, \dots , c_n)$. Let $k$ be an integer, $3 \leq k \leq n$. Then 
\begin{equation*}
c(\lambda_{k^{(1)}}, \b) =  |C_{S_n}(\b)| \sum_{\ell \geq k}^n c_\ell{\ell \choose k} f(k).
\end{equation*}
\end{theorem}
\begin{proof}
As all the  b.c.p. of $\a$ and $\b$ are in one $\ell$-cycle $\b_j=(b_1\;b_2 \; \dots  \;b_\ell)$ of $\b$, then the images under $\a$ of the b.c.p are in exactly one $\ell$-cycle $\b_{j'}=(b_1'\;b_2' \; \dots  \;b_\ell')$ of $\b$ (by Proposition~\ref{numberbc}). There are $\ell{\ell \choose k}f(k)$ ways to construct a bijection $\a|_{\set(\b_j)}: \set(\b_j) \rightarrow \set(\b_{j'})$ with steps 2 to 4 in Algorithm 1. Indeed, there are ${\ell \choose k}$ ways to choose the subset in Step 2; there are $f(k)$ ways to select the permutation $\tau$ in Step 3, and there are $\ell$ ways to select the first point in block $B_1 \dots  B_k$ in Step 4. By using Proposition~\ref{partecentral} (with $H=\{\ell\}$, $h_\ell=1$ and $r_0=\ell{\ell \choose k}f(k)$) and after summing over all possible lengths $\ell \geq k$  of cycles of $\b$ we have  
\begin{eqnarray*}
c(\lam_{k^{(1)}}, \b)= \sum_{\ell \geq k}^n  \ell{\ell \choose k}f(k) \frac{1}{\ell} c_\ell|C_{S_n}(\b)|=|C_{S_n}(\b)|  \sum_{\ell \geq k}^n c_\ell{\ell \choose k}f(k).
\end{eqnarray*}
 \end{proof}
Let $T(k, n)$ denote the number of permutations that $k$-commute with an $n$-cycle. 
\begin{corollary}\label{t1}
Let $n$ be a positive integer and $k$ and integer with $0 \leq k \leq n$.  Then
\[
T(k, n)  =  n{{n}\choose{k}}f(k).
\]
\end{corollary}
The number $T(k, n)$ is now sequence A233440 in \cite{oeis}. With this corollary we can obtain in an easy way the binomial transform of sequence A000757. Let $A=\{f(0), f(1), \dots\}$ be sequence A000757,  and let $B=\{b_0, b_1, \dots\}$ be the binomial transform of $A$. In \cite{spivey} $b_n$ is defined as $\sum_{k=0}^n {{n} \choose{k}}f(k)$ that is equal to $\sum_{k=0}^n T(k, n)/n$ by Corollary~\ref{t1}. Then we have that  $b_n=(n-1)!$ because $\sum_{k=0}^nT(k, n)=n!$. No we present the following limit property for $T(k, n)$.

\begin{proposition}
Let $n$ be a positive integer and $m$ be a fixed nonnegative integer with $m\neq n$.  Then \[\lim_{n \to \infty} \frac{T(n-m,\;  n)}{n!}=\frac{e^{-1}}{m!}.\]
\end{proposition}
\begin{proof}
\begin{eqnarray*}
T(n-m, n)/n!&=&n {n \choose n-m} f(n-m)/n!\\
&=& \frac{n}{m!(n-m)!}f(n-m)\\
&=& \frac{n}{m!}\frac{f(n-m)}{(n-m)(n-m-1)!}\\
&=& \frac{f(n-m)}{m!(n-m-1)!}+\frac{mf(n-m)}{m!(n-m)(n-m-1)!}.
\end{eqnarray*}
And the result follows by using that $\lim_{k \to \infty} f(k)/(k-1)!=e^{-1}$ (\cite[exercise 8-e, p.\ 88]{stanley}).
\end{proof}
If $m=0$ we have $\lim_{n \to \infty} \frac{T(n,\; n)}{n!}=e^{-1}$. Now we present a bivariate generating function for $T(k, n)$.

\begin{theorem}\label{nciclo}
Let $n, k$ be  positive integers with $k \leq n$. Then
\[\sum_{n,k} T(k, n)\frac{z^n}{n!}u^k=ze^{z(1-u)}\Big(\big(1-\log(1-zu)\big)\big(1-u\big)+\frac{u}{1-zu}\Big)\]
\end{theorem}
\begin{proof}
Let $g^{\langle k \rangle }(z)=\sum_n g_{n, k} \frac{z^n}{n!}$ denotes the vertical generating function (exponential case) of the sequence $\{g_{n, k}\}$. Let $c_{n, k}= T(k, n)/n={n \choose k} f(k)$. From Example III. 1,  in \cite[p. 155]{ac}, and by using the fact that function $f(k)$ is independent of $n$ we have
\[
c^{\langle k \rangle}(z)=\sum_n {n \choose k} f(k)\frac{z^n}{n!}=f(k) \frac{e^zz^k}{k!}
\]
Now, by using Rule (2') in \cite[p.\ 41] {wilf} we obtain 
\[
\sum_n n {n \choose k} f(k)\frac{z^n}{n!}=f(k) z\Big(\frac{e^zz^k}{k!}+\frac{e^zz^kk}{zk!}\Big)
\]
Now 
\begin{eqnarray*}
P(z, u)&:=&\sum_{k, n}n{n \choose k}f(k) \frac{z^n}{n!}u^k\\
&=&\sum_k \Big(\sum_n n{n \choose k} f(k)\frac{z^n}{n!}\Big)u^k\\
&=&\sum_k  f(k) z\Big(\frac{e^zz^k}{k!}+\frac{e^zz^kk}{zk!}\Big)u^k\\
&=&\sum_k  f(k) z\Big(\frac{e^zz^k}{k!}\Big)u^k+\sum_k  f(k) z\frac{e^zz^kk}{zk!}u^k\hspace{1cm} \\
&=&ze^z\sum_k  f(k) \frac{z^ku^k}{k!}+e^z\sum_k k f(k) \frac{z^ku^k}{k!}.
\end{eqnarray*}
It is known that $\sum_{k\geq 0}   f(k) \frac{x^k}{k!}=e^{-x}(1-\log(1-x))$ (see, e.g., \cite[exercise 8, p.\ 88]{stanley}), then  
\[
ze^z\sum_k  f(k) \frac{z^ku^k}{k!}=ze^z(e^{-zu}(1-\log(1-zu)).
\]
Now, we apply Rule (2') in \cite{wilf} to the second term of $P(z, u)$ to obtain
\[
e^z\sum_k k f(k) \frac{z^ku^k}{k!}=e^z(zu)e^{-zu}\Big(\frac{1}{1-zu}-(1-\log(1-zu))\Big),
\]
and the result follows after some algebraic manipulations.

\end{proof}

\subsection{On the number $c(\lam_{k^{(k)}},\b)$}
In some cases, the number $c(\lambda_k, \b)$ can be zero as shows the following  

\begin{proposition}\label{c11}
Let $\b \in S_n$ and $k$ be a positive integer. Then $c(\lam_{k^{(k)}},\b)=0.$
\end{proposition}
This result is a direct consequence of the following. 

\begin{proposition}\label{notone2}
Let $\a, \b \in S_n$.  If one cycle of $\b$ has exactly one b.c.p. of $\a$ and $\b$, then there exist a cycle of $\b$ that contains at least two b.c.p. of $\a$ and $\b$. 
\end{proposition}
\begin{proof}

The proof is by induction on the length $l$ of the cycle $\b_1$ of $\b$ which contains exactly one b.c.p of $\a$ and $\b$. First we prove the case $l=1$. By hypothesis, $\b$ has a fixed point that is a  b.c.p. of $\a$ and $\b$ and then, by Proposition~\ref{numberbc}, there exists one fixed point, say $x'$, of $\b$ whose preimage under $\a$, $\a\inv(x')$, is a b.c.p. From Proposition~\ref{transform} it follows that $\a\inv(x')$ is a point in a cycle $\b_{j}$ of $\b$ of length greater than one. Then $\b_{j}=(\a\inv(x')B)$ with $B$ a block of length $|B|\geq 1$. Therefore, $\a \b_{j} \a\inv=\big(x'B'\big)$, and, by Theorem~\ref{1cycleblocks}, $\b_{j}$ has at least two b.c.p. of $\a$ and $\b$. 

Now we prove the case $l >1$. Let $\b_1=(d_1 \dots d_l)$ be a cycle of $\b$ with exactly one b.c.p. of $\a$ and $\b$,  that without lost of generality we can suppose that is $d_l$.  Assume by induction that the statement of the proposition is true for $r$-cycles of $\b$ which contains exactly one b.c.p of $\a$ and $\b$ with $r <l$ (notice that in general, it could be the case that no such cycles in $\b$ exist).  
 
Let $\set(\mc_l(\b))=\bigcup_{\b_j \in \mc_l(\b)}\set(\b_j)$. Let $c_l$ denote the number of $l$-cycles of $\b$. By Theorem~\ref{1cycleblocks} we have that $\a \b_1\a^{-1}=(D)$ where $D=\a(d_1) \dots  \a(d_l)$ is a proper block in an $s$-cycle of $\b$, with $s >l$, which implies that $\a^{-1}\big(\set(\mc_l(\b))\big) \neq \set(\mc_l(\b))$. Then, there exist at least one $l$-cycle $\b_2=(a_1\dots a_l)$ of $\b$ (with the possibility that $\b_2=\b_1$) which contains at least one point, says $a_1$, which has its preimage under $\a$ in one $m$-cycle, say $\b_3$, of length different than $l$. Let $r$ be an integer between $1$ and $l$ such that $\a\inv(a_1) \dots  \a\inv(a_r)$ is a block $b_1 \dots  b_r$ in $\b_3=(b_1 \dots  b_{m})$ of $\b$, with  $b_i=\a \inv(a_i)$, for $1\leq i \leq r$, and such that $\a\inv(a_1) \dots  \a\inv(a_r)\a\inv(a_{r+1 \bmod l})$ is not a block in $\b_3$. We have the following cases:\\
{\it Case I.} If $m >r$, then $\b_3$ has at least two b.c.p. of $\a$ and $\b$ (by Theorem~\ref{1cycleblocks}). \\
{\it Case II.} If $m=r$ then $r < l$ (because $m \neq l$ and $1\leq r \leq l$) and $\a \b_3 \a^{-1}=(a_1 \dots  a_m)$. As  $a_1 \dots  a_m$ is a proper block in $\b_2$ then $\b_3$ has exactly one b.c.p. of $\a$ and $\b$ (by Theorem~\ref{1cycleblocks}).  And it follows from  the inductive hypothesis that $\b$ has a cycle with at least two b.c.p. of $\a$ and $\b$.  
\end{proof}
\section{The number $c(k, \b)$ for $k=3, 4$}\label{ckb34}
In this section we present formulas for the number $c(k, \b)$, when $\b$ is any permutation of cycle type $(c_1, \dots , c_n)$ and $k=3, 4$. 
\begin{theorem}\label{formulak=3}
Let $\b$ be any $n$-permutation of type $(c_1, \dots , c_n)$. Then 
\begin{equation*}
c(3, \b)=\left(\sum_{\ell \geq 3}^n  c_\ell {{\ell}\choose{3}}+\sum_{1 \leq \ell <m\leq n}    \ell mc_\ell c_m\right)|C_{S_n}(\b)|.
\end{equation*} 
\end{theorem}
\begin{proof}
By Proposition~\ref{c11} we have that $c(\lam_{3^{(3)}}, \b)=0$. Then $c(3, \b)=c(\lam_{3^{(1)}}, \b) +c([2,1], \b)$. The case  $c(\lam_{3^{(1)}}, \b)$ follows from Lemma~\ref{onlymcycles}. In order to obtain $c([2, 1],  \b)$, we construct all permutations $\a$ that $3$-commute with $\b$ and such that $\b$ has a unique $\ell$-cycle (resp. $m$-cycle), say $\b_1$ (resp. $\b_2$), where $\a$  ($1, \b$)-commutes with $\b_1$ (resp. ($2, \b)$-commutes with $\b_2$). By Proposition~\ref{numberbc} there exist exactly one  $\ell$-cycle $\b'_1$ and exactly one $m$-cycle $\b'_2$ such that $\a\left(\set(\b_1) \cup \set(\b_2)\right)=\set(\b_1') \cup \set(\b_2')$. From Theorem~\ref{1cycleblocks} 
we have that
\begin{equation}\label{eq3-1}
\a|_{\set(\b_1)\cup \set(\b_2)}=\left(
\begin{array}{cccccccc}
A_1   \\
X_1   
 \end{array}\right)\left(
\begin{array}{cccccccc}
  B_1 & B_2  \\
 X_2 &X_3  
 \end{array}\right),
\end{equation}
where 
\begin{enumerate}
\item[a)] $\b_1=(A_1)$, $\b_2=(B_1B_2)$, and $X_1, X_2, X_3$ are blocks in $\b'_1$ and $\b'_2$.
 \item[b)] The strings $X_2X_3$ and $X_3X_2$ are not blocks in any cycle of $\b$, and $X_1$ is a block in a cycle of length greater that $A_1$, i.e., $X_1$ is not a block in $\b_1'$.
 \item[c)] The set of all points in $X_1, X_2, X_3$ should be equal  to $\set(\b'_1)\cup \set(\b'_2)$. 
\end{enumerate}
From conditions a) to c) above we have that $X_2$ and $X_3$ belongs to different cycles. Without lost of generality we can assume that $\b_1'=(X_2)$ and that $\b_2'=(X_1X3)$. Now we count the number of ways to construct  $\a|_{\set(\b_1)\cup \set(\b_2)}$ using notation~(\ref{eq3-1}). There are  $\ell$ ways to select the first point of block $A_1$ and there are $m$ ways to select the first point of block $B_1B2$. There are $\ell$ ways to select the first point of block $X_2$ and there are $m$ ways to select the first point of  block $X_1$ (after this selection the first point of block $X_3$ is uniquely determined). Now by Proposition~\ref{partecentral} (with $H=\{\ell, m\}$, $h_\ell=1, h_m=1$, $r_0=(\ell m)^2$) we have $\ell mc_\ell c_m|\cen(\b)|$ for this $\ell$ and $m$. The result is obtained after we sum over all possible values of $\ell$ and $m$.
\end{proof}
\begin{theorem}\label{casek=4}
Let $\b$ be any permutation of cycle type $(c_1, \dots , c_n)$. Then 
\[
c(4, \b)=c(\lam_{4^{(1)}}, \b) +c([3, 1], \b)+c([2,2], \b)+c([2, 1,1], \b), 
\]
where 
\begin{eqnarray*}
c(\lam_{4^{(1)}}, \b) &=&|C_{S_n}(\b)| \sum_{i \geq 4}c_i {i \choose 4};\\
c([3, 1], \b)&=&|C_{S_n}(\b)| \sum_{i \geq 1, j \geq i+2} ij(j-i-1)c_ic_j;\\
c([2, 2], \b)&=&|\cen(\b)|\Big(\sum_{i\geq 2}i{i \choose 2}{c_i \choose 2}+\sum_{j>i\geq 2}i(i-1)jc_ic_j\Big);\\
c([2, 1,1], \b)&=&|\cen(\b)|\Big(\sum_{i \geq 1} i^3c_{2i}{c_i \choose 2} + \sum_{j>i\geq 1}ij(i+j)c_ic_jc_{i+j} \Big).
\end{eqnarray*}

\end{theorem}
\begin{proof}
From Proposition~\ref{c11} it follows that $c(\lam_{4^{(4)}},\b)=0$ and $c(\lam_{4^{(1)}}, \b)$ is obtained by Theorem~\ref{onlymcycles}.  We divide the rest of the proof into three parts.

{\bf Part A:} The number  $c([2, 1,1], \b)$

  Let $l_1, l_2, l_3$ denote the lengths of the cycles $\b_1, \b_2, \b_3$ of $\b$, respectively, that contains exactly two b.c.p. of $\a$ and $\b$, exactly one b.c.p., and exactly one b.c.p., respectively.  By Proposition~\ref{numberbc} $\b$ has exactly three cycles, say $\b_1', \b_2', \b_3'$, with $|\b_i'|=|\b_i|$, for $i=1, 2, 3$, such that $\a\left(\set(\b_1)\cup \set(\b_2)\cup \set(\b_3)\right)=\set(\b_1')\cup \set(\b_2')\cup \set(\b_3')$.  From Theorem~\ref{1cycleblocks} we have 
 \begin{equation}\label{eq4-1}
\a|_{\set(\b_1)\cup \set(\b_2)\cup \set(\b_3)}=\left(
\begin{array}{cccccccc}
A_1  & A_2 \\
X_1 & X_2  
 \end{array}
 \right)\left(
\begin{array}{cccccccc}
B  \\
X_3 
 \end{array}
 \right)\left(
\begin{array}{cccccccc}
 C  \\
 X_4 
 \end{array}
 \right),
\end{equation}
 
 where
 \begin{enumerate}
 \item[a)]  $\b_1=(A_1A_2)$, $\b_2=(B)$, $\b_3=(C)$, and the blocks $X_1, X_2, X_3, X_4$ came from $\b'_1$, $\b'_2$ and $\b'_3$ (not necessarily in this order), 
 \item[b)] The strings $X_1X_2$ and $X_2X_1$ are not blocks in any cycle of $\b$ and the cycle that contains $X_3$ (resp. $X_4$) is of length grater that $l_2$ (resp. $l_3$).
 \item[c)] The set of all points in $X_1, X_2, X_3, X_4$ should be equal  to $\set(\b'_1)\cup \set(\b'_2)\cup \set(\b'_3)$.
 \end{enumerate}
From conditions a) to c) above follows that the blocks $X_3$ and $X_4$ are blocks in $\b_1'$ (which implies that $l_1=l_2+l_3$) and that $X_1$ and $X_2$ are improper blocks in $\b_2'$ and $\b_3'$  (not necessarily in this order). We have two cases:
\begin{claim}\label{claim211}
\begin{enumerate}
\item If $l_2 =l_3$ then the number of permutations  in $C([2, 1,1], \b)$ satisfying this extra condition is
\[
|C_{S_n}(\b)| \sum_{l_2 \geq 1}{l_2}^3c_{2l_2}{c_{l_2} \choose 2}.
\]
\item  If $l_2 \neq l_3$ then the number of permutations in $C([2, 1,1], \b)$ satisfying this extra condition is
\[
|\cen(\b)|\sum_{l_3 >l_2 \geq 1}(l_2+l_3)l_2l_3c_{l_2+l_3}c_{l_2}c_{l_3}, 
\]
where we are assuming, without lost of generality, that $l_3 >l_2$.
\end{enumerate}
\end{claim}
\begin{proof}

\textit{1).}
  As $|\b_2'|=|\b_3'|$ then it follows that $l_1=2l_2$. Now we count all bijections using  matrix notation~\ref{eq4-1}. There are $l_1$, $l_2$, and $l_2$ ways to select the points that will be the first points of the blocks $A_1A_2, B, C$, respectively. There are $l_2$ and $l_2$ ways to select the first element of the block $X_1$ and $X_2$, respectively. There are $l_1$ ways to select the first point of the block $X_3$. Once the first point of $X_3$ is selected the first point of the block $X_4$ is uniquely determined.  
 From Proposition~\ref{partecentral} (with $H=\{l_1, l_2\}, h_{l_1}=1, h_{l_2}=2, r_0=l_1^2l_2^4$) we have that the number of permutations $\a$ that satisfies the desired conditions is $l_2^3c_{2l_2}{c_{l_2} \choose 2}|\cen(\b)|$. Finally, the result follows by summing over all possible lengths of cycles $\b_1, \b_2, \b_3$ 
\textit{Proof of 2).} $l_2 \neq l_3$. Without lost of generality we can assume that $l_3 >l_2$.  
The enumeration of all bijections as in equation~\ref{eq4-1} is similar to the previous case. Then, from Proposition~\ref{partecentral}  (with $H=\{l_1, l_2, l_3\}, h_{l_i}=1$, for every $i$, and $r_0=l_1^2l_2^2l_3^2$)  we have that the desired number is $(l_2+l_3)l_2l_3c_{l_2+l_3}c_{l_2}c_{l_3}|\cen(\b)|$ and the result follows by summing over all possible lengths of cycles $\b_1, \b_2, \b_3$. 
 \end{proof}

{\bf Part B:} The number  $c([2, 2], \b)$

 Let $\b_1$ and $\b_2$ the cycles of $\b$ of lengths $l_1, l_2$, respectively, that will contains $2$ b.c.p of $\a$ and $\b$ every one. By Proposition~\ref{numberbc}, there exist exactly two cycles $\b'_1$ and $\b'_2$ such that $\a(\set(\b_1) \cup \set(\b_2))=\set(\b'_1)\cup \set(\b'_2)$, with $|\b'_1|=|\b_1|$ and $|\b'_2|=|\b_2|$. By Theorem~\ref{1cycleblocks}) we have that
\begin{equation}\label{eq4-2}
\a|_{\set(\b_1)\cup \set(\b_2)}=\left(
\begin{array}{cccccccc}
A_1  & A_2  \\
X_1 & X_2 
 \end{array}\right)\left(
\begin{array}{cccccccc}
B_1 & B_2  \\
X_3 & X_4 
 \end{array}\right),
\end{equation}
where 
\begin{enumerate}
\item[a)] $\b_1=(A_1A_2)$, $\b_2=(B_1B_2)$, and $X_1, X_2, X_3, X_4$ are blocks in $\b'_1$ and $\b'_2$.
 \item[b)] The strings $X_1X_2$ (resp. $X_3X_4$) and $X_2X_1$ (resp. $X_4X_3$) are not blocks in any cycle of $\b$.
 \item[c)] The set of all points in $X_1, X_2, X_3, X_4$ should be equal  to $\set(\b'_1)\cup \set(\b'_2)\cup \set(\b'_3)$. 
\end{enumerate}
From points a) to c) above we have that $X_1$ and $X_2$ (resp. $X_3$ and $X_4$) belongs to different cycles. Without lost of generality we can assume that $\b_1'=(X_1X_3)$ and that $\b_2'=(X_2X4)$.

\begin{claim}\label{claim22}
 \begin{enumerate}
\item  If $l_1 =l_2$, then the number $c([2, 2], \b)$ with this extra condition is
\[
|C_{S_n}(\b)|\sum_{l_1\geq 2}l_1{l_1 \choose 2} {c_{l_1} \choose 2}.
\]
\item If $l_1 \neq l_2$, then the number $c([2, 2], \b)$ with this extra condition is

\[
|\cen(\b)|\sum_{l_2>l_1\geq 2}l_1(l_1-1)l_2c_{l_1}c_{l_2}.
\]
\end{enumerate}
\end{claim}
\begin{proof}[Proof of Claim~\ref{claim22}]
{\it (1)}  $l_1 =l_2$: We count all possible bijections $\a|_{\set(\b_1)\cup \set(\b_2)}$. There are $l_1$ and $l_1$ ways to select the first point of  block $A_1A_2$ and $B_1B_2$, respectively. There are $l_1$ ways to select the point that will be the first point of $X_1$ and $l_1-1$ ways to select the first point of $X_3$. After that, the lengths of blocks $A_1$ and $B_1$ are unique determined (and hence also the lengths of blocks $A_2$ and $B_2$).  There are $l_1$ ways to select the first point of $X_2$ and the first point of $X_4$ is uniquely determined. Then from Proposition~\ref{partecentral} (with $H=\{l_1\}, h_{l_1}=2$ and $r_0=l_1^4(l_1-1)$)  it follows that the number of permutations $\a$ satisfying the desired conditions is $l_1{c_{l_1} \choose 2}{l_1 \choose 2}|\cen(\b)|$. Finally, the result follows by summing over all possible values of $l_1$.

{\it (2)} $l_1 \neq l_2$: without lost of generality we can assume that $l_2 > l_1$. Notice that $l_1\geq 2$.  We make the enumeration of all possible bijections $\a|_{\set(\b_1)\cup \set(\b_2)}$. There are $l_1$, $l_2$ ways to select the first points of $A_1A_2$ and $B_1B_2$, respectively. There are $l_1$ ways to select the point that will be the first point of block $X_1$ and $l_1-1$ ways to select  the first point  of block $X_3$. After that, the lengths of the blocks $X_2$ and $X_4$ are uniquely determined, so that it is enough to select the first point of block $X_2$ (in $l_2$ ways). From Proposition~\ref{partecentral} (with $H=\{l_1, l_2\}, h_{l_1}=h_{l_2}=1$ and $r_0=l_1^2(l_1-1)l_2^2$)   we have that the number of permutations with the desired characteristics is $l_1(l_1-1)l_2c_{l_1} c_{l_2} |\cen(\b)|$. The result follows after summing over all possible values of the cycle lengths.
\end{proof}

{\bf Part C:}
\[c([3, 1], \b)=|C_{S_n}(\b)| \sum_{i \geq 1, j \geq i+2} ij(j-i-1)c_ic_j.
\]

Let $\a$ be any permutation that $4$-commutes with $\b$ and such that $(3, \b)$-commutes with an $l_1$-cycle, say $\b_1$, of $\b$ and that $(1, \b)$-commutes with an $l_2$-cycle, say $\b_2$, of $\b$.  By Proposition~\ref{numberbc}, there exist exactly two cycles $\b'_1$ and $\b'_2$ such that $\a(\set(\b_1) \cup \set(\b_2))=\set(\b'_1)\cup \set(\b'_2)$, with $|\b'_1|=|\b_1|$ and $|\b'_2|=|\b_2|$. From Theorem~\ref{1cycleblocks}, and without lost of generality, we have 
\begin{equation}\label{eq4-3}
\a|_{\set(\b_1)\cup \set(\b_2)}=\left(
\begin{array}{cccccccc}
A_1  & A_2 & A_3 \\
X_1 & X_2 &X_3   
 \end{array}\right)\left(
\begin{array}{cccccccc}
 B  \\
Y  
 \end{array}\right),
\end{equation}
where $\b_1=(A_1A_2A_3)$, $\b_2=(B)$, $Y$ is a block in $\b_1'$ (because $|Y|=l_2$) and $X_1, X_2, X_3$ are blocks in $\b_1'$ and $\b_2'$ such that $X_iX_{i+1}$ is not a block in any cycle of $\b$. Without lost of generality we can assume that $\b_1'=(YX_1X_3)$ which implies that $\b_2'=(X_2)$. Notice that $l_1 \geq l_2+2$. Now we make the enumeration of all possible bijections in equation~\ref{eq4-3}.  There are $l_1$, $l_2$ ways to select the points that will be the first points of blocks $A_1A_2A_3$ and $B$, respectively. There are $l_1$ ways to select the point that will be the first point of $Y$. After that, the first point of $X_1$ is uniquely determined (because $|Y|=|B|$). There are  $l_1-l_2-1$ ways to select  the first point for the block $X_3$. There are $l_2$ ways to select the first point of the block $X_2$. Then by Proposition~\ref{partecentral} we have that the desired number is 
$
l_1l_2(l_1-l_2-1)c_{l_1}c_{l_2}|\cen(\b)|.
$
Finally we obtain our result by summing over all possible lengths.

With this last case the proof of theorem is completed. 
\end{proof}

\section{Fixed-point free involutions}\label{involutions}
In this section we show formulas for $c(k, \b)$ when $\b$ is a transposition and also when it is a fixed-point free involution. First we prove the following proposition. Let $\rm{fix}(\b)$ denotes the set of fixed points of $\b$ and $\rm{supp}(\b)=[n] \setminus \rm{fix}(\b)$. 

\begin{proposition}\label{boundk2}
Let $\a, \b \in S_n$ and let $H(\a\b, \b\a)=k$, then $0 \leq k \leq 2 | \rm{supp}(\b)|$.
\end{proposition}

\begin{proof}
If $\a$ and $\b$ commute then $k=0$. If $\b$ does not have fixed points then $|\rm{supp}(\b)|=n$ and $k < 2|\rm{supp}(\b)|$. Now, let $x \in \rm{fix}(\b)$. If $\b\a(x) \neq \a\b(x)$ then $\a(x) \in \rm{supp}(\b)$ (Theorem~\ref{1cycleblocks}). Thus, $\a$ does not commute with $\b$ on at most $| \rm{supp}(\b)|$ fixed points of $\b$ and then $k \leq 2 | \rm{supp}(\b)|$.  
\end{proof}
The following proposition is a consequence of Proposition~\ref{12},  Theorem~\ref{formulak=3}, Theorem~\ref{casek=4} and Proposition~\ref{boundk2}.  

\begin{proposition}\label{transpo}
Let $\b \in S_n$ be any transposition. Then
\begin{enumerate}
\item $c(0, \b)=2(n-2)!$, $n >1$.
\item $c(3,\b)=4(n-2)(n-2)!$, $n >1$.
\item $c(4,\b)=(n-2)(n-3)(n-2)!$, $n >2$.
\item $c(k, \b)=0$, for $5\leq k\leq n$.
\end{enumerate}
\end{proposition}

We have noted that formulas  (1), (2) and (3) in previous proposition coincide with the number of permutations of $n>1$ having exactly $2$, $3$ and $4$ points, respectively,  on the boundary of their bounding square \cite{deutsch} (A208529, A208528 and A098916 in \cite{oeis}, respectively).

Now we give a formula for $c(k, \b)$ when $\b$ is any fixed-point free involution. 
 Let $a(n)$ be the ``number of derarenged matchings of $2n$ people with partners (of either sex) other than their spouse" (taken from the Comments for A053871 in~\cite{oeis}).

\begin{theorem}\label{fpfinvol}
Let $\b \in S_{2m}$ be a fixed-point free involution, $m \geq 2$. Then 
\begin{enumerate}
\item $c(k, \b)=0$, for $k$ and odd integer,
\item $c(k, \b)=2^{m}m!{m \choose j}a(j)$, for $k=2j$, $j=0, 1, 2, \dots$
\end{enumerate}
\end{theorem}
\begin{proof}
 From Proposition~\ref{corone} we have that if $\a$ does not commute on one point in cycle $\b_i$ of $\b$ then $\a$ does not commute on the two points in $\b_i$, then if $k$ is odd we obtain the desired result.  Now we will obtain all the permutations $\a$ that $k$-commutes with $\b$ and  that do not commute on $j$ $2$-cycles, $\b_1, \b_2,  \dots,\b_j$,  of $\b$. By Proposition~\ref{numberbc} there exists exactly $j$ $2$-cycles, $\b_1', \b_2', \dots, \b_j'$ of $\b$ such that $\a(\bigcup_{i=1}^j\set(\b_i))=\bigcup_{i=1}^j\set(\b_i')$. We construct $\a|_{\bigcup_{i=1}^j\set(\b_i)}$ as follows: First put 
\[
\a'|_{\bigcup_{i=1}^j\set(b_i)}=\left(
\begin{array}{ccccccccccccccc}
  B_1   \\
B_{i_1'} \\
 \end{array}
\right)\left(
\begin{array}{ccccccccccccccc}
  B_2  \\
 B_{i_2'} \\
 \end{array}
\right)\dots \left(
\begin{array}{ccccccccccccccc}
B_j  \\
 B_{i_j'}\\
 \end{array}
\right),
\]
where $B_i$ is an improper block of the cycle $\b_i$ and $B_i'$ is an improper block of $\b_i'$, for every $i$,  and where $B_{i_1'}  B_{i_2'} \dots   B_{i_j'}$ is any block permutation of the $j$ improper blocks of $\b_1', \dots,\b_k'$ (there are $j!$ such permutations). Until this step $\a'|_{\bigcup_{i=1}^j\set(b_i)}$ is a permutation that commutes with $\b$ on the cycles $\b_1, \dots, \b_n$.  Now we can see every block $B_i'$ as a partner $B_i'=xy$, and we finish the construction of $\a|_{\bigcup_{i=1}^j\set(b_i)}$ by re-paired the elements in the blocks $B_{i_1'}, B_{i_2'},  \dots ,  B_{i_j'}$ in such a way that no point is with its original partner, this can be made in $a(j)$ ways. As there are $2^j$ ways to select the first element in the blocks $B_1, \dots, B_j$, we have that $r_0=2^j j!a(j)$. Finally, by Proposition~\ref{partecentral} (with $\ell=2$, $c_\ell=m$, $h_\ell=j$ and $|C_{S_n}(\b)|=2^mm!$) we obtain
\[
2^j j!a(j){m \choose j}^2 \frac{(m-j)!}{2^jm!}2^mm!=2^mm! {m \choose j}
a(j).
\]

\end{proof}

\begin{theorem} 
Let $\b \in S_{2m}$ be a fixed-point free involution. Then
\[
\sum_{m, j \geq 0} c(2j, \b) \frac{z^m}{m!}\frac{u^j}{j!}= \left(\left( 1-2\,z \right) \sqrt{1-4\,{\frac {zu}{1-2\,z}}} \exp \left(\frac {2zu}{1-2\,z}\right) \right)^{-1}.
\]
\end{theorem}
\begin{proof}(Sketch)
We use the well-known EGF for $a(n)$ (A053871 in~\cite{oeis})
\[
\sum_{n \geq 0} a(n) \frac{x^n}{n!}=\left(\exp(x)\sqrt{1-2x}\right)^{-1},
\]
and the result follows by using standard techniques of bivariate generating functions similarly as in the proof of Theorem~\ref{nciclo}.
\end{proof}

\section{Final comments}
 As we have seen in previous sections, the problem of computing in an exact way the number $c(k, \b)$ can be a difficult task. This is because we have that the number of cases can be equal to the number of partitions $[k_1, \dots  , k_h]$ of $k$, with $k_i\geq 1$. Even more, once we have selected a partition of $k$, it can be difficult  to compute in exact way all the permutations with the desired properties. For example, if we have that $\b$ has at least two cycles $\b_1$ and $\b_2$ of lengths $l$ and $m$, respectively, and we look for all the permutations that $k$-commute with $\b$ and that $(k_1, \b)$-commutes with $\b_1$ and that $(k_2, \b)$-commutes with $\b_2$ then $\a|_{\set(\b_1) \cup \set(\b_2)}$ must look like
 \begin{equation*}
\a|_{\set(\b_1)\cup \set(\b_2)}=\left(
\begin{array}{cccccccc}
A_1  & \dots  &A_{k_1}   \\
X_1' & \dots  &X_{k_1}'  
 \end{array}\right)\left(
\begin{array}{cccccccc}
 B_1 & \dots  & B_{k_2}   \\
 Y_1' & \dots  & Y_{k_2}'  
 \end{array}\right),
\end{equation*}
where we can have many possibilities for the selection of blocks $X_1', \dots , X_{k_1}', Y_1', \dots , Y_{k_2}'$, that depends  of the lengths of the cycles $\b_1, \b_2$. And after we have selected these blocks, we have the problem of the number of ways in which we can arranged it. However, it is possible that for some specific cycle type of permutations, the problem can be managed. We leave as an open problem to find another technique, or a refinement of the presented in this article, to compute $c(k, \b)$ in exact way, or at least to obtain non trivial upper and lower bounds for this number.

\section{Acknowledgements}
The authors would like to thank L. Glebsky for very useful suggestions and comments. The authors also would like to thank Jes\'us Lea\~nos for his careful reading of the paper and his very valuable suggestions. The second author was supported by the European Research Council (ERC) grant of Goulnara Arzhantseva, grant agreement No. 259527 and by PROMEP (SEP, M\'exico) grant UAZ-PTC-103 (No. 103.5/09/4144
and No. 103.5/11/3795)

\end{document}